\documentclass[12pt]{amsart}
\usepackage{amscd}
\usepackage{epsf}
\usepackage{epsfig}
\newtheorem{theorem}{Theorem}[section]
\newtheorem{proposition}[theorem]{Proposition}
\newtheorem{lemma}[theorem]{Lemma}
\newtheorem{corollary}[theorem]{Corollary}
\theoremstyle{remark} 
\begin{document}

\markboth{I Biswas, S Chatterjee, R Dey}
{Geometric prequantization on
the path space of a prequantized manifold}

%
%

\title{Geometric prequantization on
the path space of a prequantized manifold}

\author{INDRANIL BISWAS}

\address{School of Mathematics, Tata Institute of Fundamental
Research,\\ Homi Bhabha Road, Mumbai 400005, India\\
indranil@math.tifr.res.in}

\author{SAIKAT CHATTERJEE}

\address{Department of Mathematics,  Institut des Hautes \'Etudes Scientifiques,\\ 35 Route de Chartres, Bures-sur-Yvette 91440, France\\saikat.chat01@gmail.com}

\author{RUKMINI DEY}

\address{School of Mathematics, Harish Chandra Research Institute,\\ Jhusi,
Allahabad 211019, India\\
rukmini.dey@gmail.com }

\maketitle

\begin{abstract}
Given a compact symplectic manifold $M$, with integral symplectic form,
we prequantize a certain class of functions on the path space for $M$.
The functions in question are induced by functions on $M$. We apply our construction 
to study the symplectic structure on the solution space of Klein-Gordon equation.
\end{abstract}


\section{Introduction}\label{s:int}

Let $M$ be a compact symplectic manifold with a symplectic form
$\omega$ such that the cohomology class of $\omega$ lies in the
image of $H^2(M,\, 2\pi\hbar\cdot {\mathbb Z})$ in
$H^2(M,\,{\mathbb R})$. Then one knows that there is a prequantum line bundle $L$ with a 
Hermitian metric, and a connection whose curvature is $\hbar^{-1}\cdot\omega$.
{\textit {Geometric prequantization}}  is a construction  of this 
prequantum line bundle $L$ equipped with a Hermitian connection. 
The Hilbert space of the prequantization is
the space of square integrable sections of $L$. Here in general we
assune manifold $M$ to be compact, however  many of
our constructions are valid without that assumption.

Fix a Hermitian connection $\nabla$ on $L$ whose curvature is
$\hbar^{-1}\cdot\omega$. To every $f \in C^{\infty}( M)$, we associate an
operator ${f}_{\rm op} \,:=\, -\sqrt{-1}\hbar\nabla_{X_f} + f $
acting on the above Hilbert space, where
$\nabla_{X_f}$ is the covariant derivative along the Hamiltonian vector field $X_f$ corresponding to $f$. 
For $f_1, f_2 \in C^{\infty} (M)$, let 
$f_3\, :=\, \{ f_1, f_2 \}$ be the corresponding Poisson bracket
induced by the symplectic form $\omega$. Then $[{f}_{1\, {\rm op}}, {f}_{2\,{\rm op}}]\,=\,
 -{\sqrt{-1}} \hbar {f}_{3\, {\rm op}}$ ~\cite{Wo}. This is the most important
of the quantization conditions.

Let ${\mathcal P} {M}$ denote the path space for $M$; more precisely, $\mathcal P {M}$ is the
space of all smooth  maps $\gamma:[0, 1]\longrightarrow M$. 
Any $f \in C^{\infty}( M)$ defines a $C^{\infty}$ function on $\mathcal P {M}$ by
\begin{equation}\label{intro:eq}
{\widetilde f}(\gamma):=\int_0^1f(\gamma(t))dt.
\end{equation}
Given a symplectic form $\omega$ on $M$, in \eqref{sympath} we construct a 
weakly symplectic form $\widetilde \omega$ on $\mathcal P {M}$. 
We observe here that for the construction of the symplectic 
form $\widetilde \omega$, the compactness condition is not necessary. Since, the symplectic form
is weak, an arbitrary $C^{\infty}$ function on $\mathcal P {M}$ may not have a corresponding 
Hamiltonian vector field. For that reason we consider the space of 
functions (on path space) to be generated by finite polynomials of functions
of the type in \eqref{intro:eq}; that is functions on $\mathcal P {M}$ which can be written as
\begin{equation} \nonumber
\sum_i c_i \cdot (\widetilde{f_i})^{n_i} .
\end {equation}
Corollary~\ref{c:Hamilentire} describes the Hamiltonian vector fields on the space of 
functions defined above. We refer to \cite{mars, bam, mc} for a detail discussion about
Hamiltonian vector fields on weakly symplectic spaces. In this paper, we further 
exhibit a Poisson structure on that class of functions on $\mathcal P {M}$. The relationship between
the Poisson structure on $M$ and the Poisson structure defined by the symplectic form 
$\widetilde \omega$ on $\mathcal P {M}$ is investigated.

We show that the symplectic form $\widetilde \omega$ can be expressed as the sum of
the pull-back of the symplectic form $\omega$ and the exterior derivative
of a certain globally defined $1$-form $\lambda$ on $\mathcal P {M}$. Consequently, we obtain a
Hermitian connection $\widetilde \nabla$, given in \eqref{conpullback},
 on the prequantum bundle $ {\rm ev}_0^* L\,\longrightarrow\,
\mathcal P {M}$ whose curvature is  $\hbar^{-1}\widetilde \omega$. We emphasize here that while the line 
bundle over $\mathcal P {M}$ is constructed by pulling back the line bundle $L$ on $M$
via ${\rm ev}_0$, the connection $\widetilde \nabla$ is not simply the pull-back
of connection $\nabla$. The construction of the line bundle on the loop space and 
its relation to geometric quantization has been discussed in \cite{bryl, gawd} and others. 
In \cite{bryl} the line bundle on the loop space ${\mathcal L}M$ has been constructed exploiting the 
connection structure on the sheaf of groupoids
over the manifold $M$ and the curvature of that line bundle is obtained
by  Chen integral of a``curvature $3$-form"  on $M$ [Proposition 6.2.2, \cite{bryl}]. 
Whereas in our construction the curvature of the line bundle over $\mathcal P {M}$ is given by an 
``integration" of the symplectic $2$-form $\omega$ on $M$. However, we will see in 
Proposition~\ref{p:tildenabla} that the $1$-form $\lambda$ in our path space connection 
$\widetilde \nabla$ is obtained by Chen integral of the $2$-form $\omega$ on $M$.

There exists a certain special class of sections of the line bundle 
${\rm ev}_0^* L$ given by the 
pull-backs of square integrable sections  of the prequantize line bundle $L$,
which defines a Hilbert space. Let $\phi$ be a function on $\mathcal P {M}$ induced by a functions on $M$
We can associate to it an operator ${\phi}_{\rm op}$ as follows:
$$
{\phi}_{\rm op}\,:=\, -\sqrt{-1}\hbar{\widetilde \nabla}_{{\widetilde X}_{\phi}}+
{\phi}\, ,
$$
where ${\widetilde \nabla}$ is the above mentioned connection on the line bundle
${\rm ev}_0^* L$ over ${\mathcal P}M$. These operators satisfy the prequantization 
condition, because the curvature of the connection is $\hbar^{-1}\widetilde \omega$. Thus we have 
prequantized a certain class of functions on the path space of $M$.

The final section is devoted to an application of our construction. Let $M$ be the
space-time manifold with a pseudo-Riemannian metric, and let
${\overline M}\,:=\, M\times S^1$ be the manifold 
with an additional spatial compact dimension $S^1.$ Introduction of additional compact manifolds, popularly known as ``compactification" in physics,
 plays a pivotal role in Kaluza-Klein theories and string theory~\cite{G-S-W, Gr}. The spaces of fields on manifolds, which are solutions 
of certain physical theories, such as (real or complex) scalar fields, Yang-Mills fields, electromagnetic fields, have
symplectic structures~\cite{Wo}.

Though here we exclusively focus on real scalar fields given by the solutions of 
Klein-Gordon equation, our construction can be generalized for other theories as well. Let $\mathcal M$ be the space of (real) scalar fields, satisfying 
Klein-Gordon equation, on $M$, and let ${\overline {\mathcal M}}$ be likewise for $\overline M.$ 
The space  $\mathcal M$ or ${\overline {\mathcal M}}$ is obviously not compact. However, since in order
to construct the symplectic form on path space, we do not need to assume the compactness. The following
goes through. We first construct the path space ${\mathcal P}{\mathcal M}$ over $\mathcal M$ and following our 
constructions in preceding sections we define a symplectic form on  ${\mathcal P}{\mathcal M}$. 
We show that $\overline {\mathcal M}$ can be identified with a subspace of ${\mathcal P}{\mathcal M}.$ 
Proposition~\ref{p:coinsym} proves that the symplectic form on  $\overline {\mathcal M}$, given by the 
symplectic structure as solution space of Klein-Gordon equation, coincides with the symplectic form 
defined by the symplectic structure on the path space of $\mathcal M$.

\section{Poisson structure on Path space}\label{s:poiss}

Let $I\,:=\,[0\,,1]\,\subset\, {\mathbb R}$ be the closed interval, and let
$M$ be a smooth real manifold of finite dimension. 
The \textit{path space} of $M$, which is denoted by ${\mathcal P} {M}$,
is the space of all smooth differentiable 
maps $$\gamma \,:\, I \,\longrightarrow\, M\, .$$ For any $\tau\, \in\,
[0\,,1]$, the restriction of a path $\gamma$ to
the sub-interval $[0\, , \tau]\,
\subset\, I$ will be denoted by $\gamma_{\tau}$. This
$\gamma_{\tau}$ is called a \textit{segment} of $\gamma$.
The path space $\mathcal P {M}$ has a differentiable structure~\cite{chen, chen1}. 
The tangent space to $\mathcal P {M}$ at $\gamma$ is the
space of all smooth vector fields $t\,\longmapsto\, T_{\gamma(t)}M$,
$t\,\in\, I$.

Let 
\begin{equation}\label{evm}
{\rm ev}\,:\,\mathcal P {M} \times [0\, , 1]\,\longrightarrow\, M\, , ~\
(\gamma\, , t)\,\longmapsto\, \gamma(t)\, ,
\end{equation}
be the evaluation map. For any $t\, \in\, I$, define
$$
{\rm ev}_t\,:\, \mathcal P {M} \, \longrightarrow\, M\, , ~\
\gamma\,\longmapsto\, {\rm ev}(\gamma,t)\,=\,\gamma(t)\, .
$$
For any $C^{\infty}$ function $F$ on $\mathcal P {M} \times [0, 1]$, define a function
${\overline F}$ on $\mathcal P {M}$ by
$$
\overline{F} (\gamma) :\,= \, \int_{0}^{1} F (\gamma, t) dt\, .
$$
Now, given any function
$f \,\in\, C^{\infty} \bigl(M\bigr)$,  we have the function
${\widehat f}\,:=\,{\rm ev}^{*}f$ on ${\mathcal P}  M \times [0\, ,1]$
defined by 
$$
\widehat{f}(\gamma, t) \,:=\,  f (\gamma(t))\, .
$$
Let
\begin{equation} \label{funcpm}
\widetilde{f} (\gamma) :\,= \, \overline{\widehat{f}}\,=\, 
\int_{0}^{1} \widehat{f} (\gamma, t) dt \,=\, \int_{0}^{1} f (\gamma(t)) dt\, .
\end{equation} 
be the function on ${\mathcal P}  M$.

Note that the map $f\, \longmapsto\, \widetilde{f}$ is injective.
Indeed, if $\widetilde{f}\,=\, \widetilde{g}$, then for any $x\, \in\, M$,
taking $\gamma$ to be the constant map $I\, \longrightarrow\, x$, from
\eqref{funcpm} we conclude that $f(x)\,=\, g(x)$. Let 
\begin{equation}\label{om10}
\Omega_{1}^{0} {\mathcal P}{M}
\end{equation}
denote the space of all functions on $\mathcal P {M}$ defined by the 
$C^{\infty}$ functions on $M$ using \eqref{funcpm}.

Let $\Omega^{0} {\mathcal P}{M}$ denote the space of all
$C^\infty$ functions on ${\mathcal P}{M}$ that can be written as
a finite sum of the form 
\begin{equation} \label{uptok}
\sum_i c_i \cdot (\widetilde{f_i})^{n_i}\, ,
\end{equation}
where $c_i\,\in\, \mathbb R$, $f_i\, \in\, C^\infty(M)$ and $n_i$ are
nonnegative integers. Note, if $f\in C^{\infty}(M)$, then the pull-back function
${\rm ev}_{t_0}^{*}f$ on $\mathcal P {M},$ for any $t_0\in [0,1]=I$ can be written as
$${\rm ev}_{t_0}^{*}f(\gamma)=\int_0^1 f(\gamma(t))\delta(t-t_0)dt,$$
where $\delta(t-t_0)$ is the Dirac Delta function. Thus the space of functions
 on path space $\mathcal P {M}$ obtained by
pull-back of the functions on $M$ is a subset of  $\Omega^{0} {\mathcal P}{M}$.

Take $f\, ,g\, \in\, C^\infty(M)$. Observe that we 
have the following identity for the product of two line integrals:
\begin{equation}\label{productline}
\left(\int_{0}^{1} f (\gamma(t)) dt\right)\left(\int_{0}^{1} g (\gamma(t)) 
dt\right)\,=
\end{equation}
$$
\int_{0}^{1}f(\gamma(t))\left(\int_{0}^{t} g(\gamma(\tau))d\tau\right)dt +
\int_{0}^{1}\left(\int_{0}^{t} f(\gamma(\tau))d\tau\right)g(\gamma(t))dt\, .
$$
Now, we construct a function 
${\widehat {f^{\bullet}}}$ on $\mathcal P {M}\times [0,1]$ as follows:
$$
{\widehat {f^{\bullet}}}(\gamma, t)\,:=\,\int_{0}^{t}f(\gamma(\tau))d\tau
$$
for all $t\in [0,1]$ and $\gamma\in \mathcal P {M}$.
Then the right--hand side of the identity \eqref{productline} is
$$
\left(\int_{0}^{1} f(\gamma(t)) dt\right)\left(\int_{0}^{1} g (\gamma(t)) dt\right)
$$
$$
=\,\int_{0}^{1}\bigl({\widehat f}\cdot {\widehat {g^{\bullet}}}\bigr)
(\gamma, t)dt +\int_{0}^{1}\bigl({\widehat {f^{\bullet}}}\cdot
{\widehat g}\bigr)(\gamma, t)dt\, .
$$
Therefore, we have
$$
({\widetilde f}\cdot {\widetilde g})(\gamma)
\,=\, \int_{0}^{1}({\widehat f}\cdot {\widehat {g^{\bullet}}})(\gamma, t)
dt +\int_{0}^{1}({\widehat {f^{\bullet}}}\cdot {\widehat g})
(\gamma, t)dt\, ,
$$
where ${\widetilde f}$ and ${\widetilde g}$ are constructed as in \eqref{funcpm}.
This defines a product law on $\Omega^0{\mathcal P {M}}$. Also if ${\widetilde f}$ is given by 
$${\widetilde f}(\gamma)\,=\,\int_{0}^{1} f (\gamma(t)) dt\, ,$$
then the de Rham differential of ${\widetilde f}$ is given by following formula~\cite{BC, chen}:
\begin{equation} \label{exti}
{{\rm d}\widetilde f (v)}(\gamma)\,=\,\int_{0}^{1} df ({d {\rm ev}_t}v(\gamma)) dt\, ,
\end{equation}
where $v\,\in \,T\mathcal P {M}$ and $d {\rm ev}_t$ be the differential of the map ${\rm ev}_t$. By
Leibniz rule, \eqref{exti} extends to $\Omega^{0} {\mathcal P}{M}.$

Suppose $M$ be a compact connected symplectic manifold equipped
with a symplectic form $\omega$. Given a $C^{\infty}$ function $f$ on $M$,
let $X_f$ be the \textit{Hamiltonian vector field} on $M$ defined by
\begin{equation}\label{hamilbasesym}
i_{X_{f}}\omega\,=\,-df\, ,
\end{equation}
where $i$ is the the contraction of differential forms by vector fields, and $d$ is
the de Rham differential. It should be clarified that for notational convenience 
we will often denote a contraction by writing 
the vector field as an argument of the differential form:
$$i_{v}\alpha (-)\,\equiv\, \alpha (v\, ,-)\, ,$$
where $\alpha$ is a differential form and $v$ is a tangent vector field on some
manifold. The $2$-form 
$\omega$ induces a Poisson structure on $M$.
The Poisson bracket of $f, g\in C^{\infty} \bigl(M\bigr)$ is defined to be
$$
\{f\, , g\}\,:=\, \omega (X_f\, , X_g)\, .
$$
Equivalently, the Poisson bracket can be written as
$$
\{f, g\}\,=\, i_{X_f}(dg)\, .
$$
When there are more than one symplectic manifolds, we will denote the Poisson
structure on $M$ by $\{, \}_M$ to avoid any confusion.

Let us define a symplectic form $\widetilde \omega$ on $\mathcal P {M}$ as follows
\begin{equation}\label{sympath}
\bigl(\widetilde \omega (v_1, v_2)\bigr)(\gamma)\,:=\,
\int_{0}^{1} \omega \bigl(d{\rm ev}_t (v_1 (\gamma)), 
d{\rm ev}_t (v_2 (\gamma))\bigr)dt\, .
\end{equation}
 
It is straightforward to verify that $\widetilde \omega$ is closed
and nondegenerate~\cite{BC}. However,  $\omega$ is symplectic only in a 
weak sense; that is the induced map 
\begin{eqnarray}
&{\widetilde \omega}^{\circ}:&T\mathcal P {M} \rightarrow T^{*}\mathcal P {M},\nonumber\\
&&v\mapsto \widetilde \omega (v,-)\nonumber
\end{eqnarray}
is only injective, not surjective. Thus as observed in Section~\ref{s:int},
 an arbitrary $C^{\infty}$ function on $\mathcal P {M}$ does not guarantee
a corresponding Hamiltonian vector field. In what follows we show
that for the class of functions defined in \eqref{uptok} Hamiltonian
vector fields do exist. Take any ${\phi}\,\in\, {{\Omega}^{0}}\mathcal P {M}$,
where ${{\Omega}^{0}}\mathcal P {M}$ is 
defined in \eqref{uptok}, and let $\widetilde{X}_{\phi}\,\in\, T \mathcal P {M}$
be the Hamiltonian vector 
field corresponding to ${\phi}$; if it exists then, it is uniquely 
determined by the equation
\begin{equation}\label{hamilpm}
i_{{\widetilde X}_{\phi}}{\widetilde \omega}\,=\, -{\rm d}{\phi}\, .
\end{equation}
We define the Poisson bracket on the path space $\mathcal P {M}$ as follows:
\begin{equation}\label{pathpoiss}
\{{\phi}_1, {\phi}_2 \}_{\mathcal P {M}}\,:=\,
{\widetilde \omega}({\widetilde X}_{\phi_1}, {\widetilde X}_{\phi_2})\, ,
\end{equation}
where ${\phi}_1, {\phi}_2\in {{\Omega}^{0}}\mathcal P {M}$ and 
${\widetilde X}_{{\phi}_1}, {\widetilde X}_{{\phi}_2}\in T\mathcal P {M}$ are their respective Hamiltonian vector fields.

Now, suppose ${\widetilde f}\in {{\Omega}^{0}_{\rm 1}}\mathcal P {M}\,\subset\, {{\Omega}^{0}}\mathcal P {M}$, 
where ${{\Omega}^{0}_{\rm 1}}\mathcal P {M}$ is defined in \eqref{om10}. Hence ${\widetilde f}$ is given by
$${\widetilde f}(\gamma)\,=\,\int_{0}^{1} f (\gamma(t)) dt$$ 
for some $f\in C^{\infty}M.$ Observe by \eqref{exti} we have
\begin{equation}\label{hamilpmexpl}
\begin{split}
&i_{{\widetilde X}_{\widetilde f}}{\widetilde \omega}\,=\, -{\rm d}{\widetilde f}\, \\
&\Rightarrow~ {\widetilde \omega}\bigl({{\widetilde X}_{\widetilde f}}\, ,
- \bigr)(\gamma)=-\int_0^{1}(d{ f}(-))_{\gamma(t)}dt\, ,
\end{split}
\end{equation}
where ${{\widetilde X}_{\widetilde f}}\in T \mathcal P {M}$ is the Hamiltonian vector field corresponding to ${\widetilde f}.$
Using the definition of $\widetilde \omega$ in \eqref{sympath}, the above equation reads
$$
\int_{0}^{1} \omega \bigl(d{\rm ev}_t ({{\widetilde X}_{\widetilde f}} (\gamma)), 
- \bigr)dt\,=\,-\int_0^{1}(d{ f}(-))_{\gamma(t)}dt
$$
for all $\gamma\,\in \,\mathcal P {M}$.
 
\begin{proposition}\label{p:identify}
Suppose ${\widetilde f}\in {{\Omega}^{0}_{\rm 1}}\mathcal P {M}\,\subset\, {{\Omega}^{0}}\mathcal P {M}$ is given by
$${\widetilde f}(\gamma)\,=\,\int_{0}^{1} f (\gamma(t)) dt
$$
where $f\in C^{\infty}M$.
Then the Hamiltonian vector field ${{\widetilde X}_{\widetilde f}}\in T \mathcal P {M}$ corresponding to ${\widetilde f}$, given by \eqref{hamilpmexpl},
satisfies 
$$
\bigl({{\widetilde X}_{\widetilde f}}(\gamma)\bigr)(t)\,=\,X_f(\gamma(t))\, ,
$$
where $X_f \in TM$ is the Hamiltonian vector field
defined in \eqref{hamilbasesym} corresponding to $f\,\in\, C^{\infty}(M)$.
\end{proposition}

\begin{proof}
Suppose $X_f \in TM$ is the Hamiltonian vector field corresponding to $f\in C^{\infty}M.$ We define a vector field  $\chi \in T\mathcal P {M}$ as follows:
$$
\bigl(\chi(\gamma)\bigr)(t)\,:=\,X_f(\gamma(t))\, .
$$
We claim that $\chi$ is the Hamiltonian vector field corresponding to ${\widetilde f}.$ As Hamiltonian vector field is unique, 
to prove the claim, it is enough to show the $\chi$ satisfies 
$$\widetilde \omega (\chi, v)\,=\,-{\rm d}{\widetilde f}(v)\, .$$
Note that by construction $$d{\rm ev}_t (\chi (\gamma))\,=\,X_f(\gamma(t))\, .$$
Hence, using the relation $\omega(X_f, v)=-df(v)$ and the definition of $\widetilde \omega$ in \eqref{sympath} we obtain
\begin{equation}\nonumber
\begin{split}
&\int_{0}^{1} \omega \bigl(d{\rm ev}_t ({\chi} (\gamma)), - \bigr)dt\\
=&\int_{0}^{1} \omega \bigl(X_f(\gamma(t)), - \bigr)dt\\
=&-\int_0^{1}(d{ f}(-))_{\gamma(t)}dt,
\end{split}
\end{equation}
for all $\gamma\,\in\, \mathcal P {M}$. Thus,
$$\widetilde \omega (\chi, v)=-{\rm d}{\widetilde f}(v)$$
proving the claim. The proposition follows from the claim.
\end{proof}

\begin{corollary}\label{c:identifypull}
If ${\widetilde f}, {\widetilde g} \in {{\Omega}^{0}_{\rm 1}}\mathcal P {M}\,\subset\,
 {{\Omega}^{0}}\mathcal P {M}$ are given by
\begin{eqnarray}
&&{\widetilde f}(\gamma)=\int_{0}^{1} f (\gamma(t)) dt, {\hskip 0.3 cm} {\rm where} {\hskip 0.2 cm} f\in C^{\infty}M,\nonumber\\
&&{\widetilde g}(\gamma)=\int_{0}^{1} g (\gamma(t)) dt, {\hskip 0.3 cm} {\rm where} {\hskip 0.2 cm} g\in C^{\infty}M,\nonumber
\end{eqnarray}
then 
$$
\{{\widetilde f}, {\widetilde g}\}_{\mathcal P {M}}(\gamma)\,=\,\int_0^{1}\{{f}, {g}\}(\gamma(t))dt\, .
$$
\end{corollary}

\begin{proof}
This follows directly from Proposition~\ref{p:identify}.
\end{proof}

\begin{corollary}\label{c:Hamilentire}
Hamiltonian vector fields, with respect to $\widetilde \omega$, on ${{\Omega}^{0}}\mathcal P {M}$
are completely determined by the Hamiltonian vector fields,
with respect to $\omega$, on $C^{\infty}(M)$. More precisely, suppose
${\zeta}\,:=\,{\widetilde f}\widetilde{g}\,\in\,
{{\Omega}^{0}}\mathcal P {M}$, where ${\widetilde f}, {\widetilde g} \in {{\Omega}^{0}_{\rm 1}}\mathcal P {M}
\,\subset\, {{\Omega}^{0}}\mathcal P {M}$ are given by
\begin{eqnarray}
&&{\widetilde f}(\gamma)=\int_{0}^{1} f (\gamma(t)) dt, {\hskip 0.3 cm} {\rm where} 
{\hskip 0.2 cm} f\in C^{\infty}M,\nonumber\\
&&{\widetilde g}(\gamma)=\int_{0}^{1} g (\gamma(t)) dt, {\hskip 0.3 cm} {\rm where} {\hskip 0.2 cm} g\in C^{\infty}M.\nonumber
\end{eqnarray}
Then the Hamiltonian vector field ${\widetilde X}_{\zeta}\in T\mathcal P {M}$ corresponding to
${\zeta}$ is given as follows: 
\begin{equation}\label{hamilsum}
\begin{split}
&\bigl({\widetilde X}_{\zeta}(\gamma)\bigr)(t)=\\ 
&\bigl(\int_{0}^{1} g (\gamma(t)) dt\bigr)\cdot X_f(\gamma(t))+
\bigl(\int_{0}^{1} f (\gamma(t)) dt\bigr)\cdot X_g(\gamma(t))\, .
\end{split}
\end{equation}
\end{corollary}

\begin{proof}
First we observe that by construction, the entire set 
${{\Omega}^{0}}\mathcal P {M}$ is generated by $C^{\infty}(M)$ (see \eqref{uptok}). 
On the other hand, we have already shown in Proposition~\ref{p:identify} that 
${\widetilde X}_{\widetilde f}$ is determined by $X_f$, 
thus it is sufficient to prove \eqref{hamilsum} for our purpose.

For ${\zeta}:={\widetilde f}\widetilde{g}$ we have the identity
$${\widetilde X}_{\zeta}\,=\,{\widetilde f}{\widetilde X}_{\widetilde g}+
{\widetilde g}{\widetilde X}_{\widetilde f}\, ,$$
which implies
$${\widetilde X}_{\zeta}(\gamma)\,=\,{\widetilde f}(\gamma)\cdot
{\widetilde X}_{\widetilde g}(\gamma)+{\widetilde g}(\gamma)\cdot
{\widetilde X}_{\widetilde f}(\gamma)\, .$$
So, by Proposition~\ref{p:identify} we have 
$$
\bigl({\widetilde X}_{\zeta}(\gamma)\bigr)(t)\,=\,
\bigl(\int_{0}^{1} f (\gamma(t)) dt\bigr)\cdot X_g(\gamma(t))+\bigl(\int_{0}^{1} g (\gamma(t)) dt\bigr)\cdot X_f(\gamma(t)).
$$
This completes the proof.
\end{proof}

Another consequence of Proposition~\ref{p:identify} is the following system of
identities for the Lie-bracket of Hamiltonian vector fields ${\widetilde 
X}_{\widetilde f}\in T\mathcal P {M}$ and ${\widetilde X}_{\widetilde g}\in T\mathcal P {M}$:
$$
\left([ {\widetilde X}_{\widetilde f}, {\widetilde X}_{\widetilde g}]
(\gamma)\right)(t)\,=\, [X_f, X_g](\gamma(t))\,=\,
\left({\widetilde X}_{\{{\widetilde f},{\widetilde g}\}}(\gamma)\right)(t)\,
=\,X_{\{f, g\}}(\gamma(t))\, ,
$$
where as usual 
$$
{\widetilde f}(\gamma)\,=\,\int_{0}^{1} f (\gamma(t)) dt\ \text{ and }\
{\widetilde g}(\gamma)\,=\,\int_{0}^{1} g (\gamma(t)) dt\, .
$$
\section{The symplectic potential} 

As before, $M$ is a compact connected manifold equipped with a symplectic form $\omega$. 
The symplectic form $\widetilde \omega$ 
on the path space $\mathcal P M$ is given  by \eqref{sympath}.
As $\omega$ is closed, it can locally be written as $d\theta$, where $\theta$ is a
$1$-form defined over some open subsets of $M$. We call $\theta$ the \textit{symplectic potential}. Obviously the choice of $\theta$ is not canonical.
In this section we will show that the choice of a particular local symplectic potentials $\theta$ for $\omega$ naturally defines a 
symplectic potential for $\widetilde \omega$ on $\mathcal P M$. 
 
Given a path $\gamma:[0,1]\longrightarrow M$ on $M$, we can partition the interval $[0,1]$ as 
$0 = t_0 < t_1 <\cdots < t_n = 1$ such that the image of the segment $\gamma\vert_{[t_{i-1}, t_i]}$  is entirely contained in a 
neighborhood  $U_i \subset M$ for each $i \in \{ 1,\cdots , n \}$, where $\omega\vert_{U_i}
\,=\, d \theta_i$. Let ${\mathcal U}$ be the set of all such paths:
$$
{\mathcal U} := \{ \gamma \in {\mathcal P}{\mathcal M}\,\mid\, {\rm Image}\left(\gamma\vert_{[t_{i - 1}, t_i ] }\right) \subset U_i, {\hskip 0.15 cm} {\rm for {\hskip 0.1 cm} each}{\hskip 0.2 cm}   i\in [1,n] \}\subset \mathcal P M.
$$
It should be mentioned that such a set exists because $M$ is compact. Now re-parametrize each interval ${[t_{i-1}, t_i]}$ 
by a $\sigma_i:{[t_{i-1}, t_i]}\longrightarrow [0,1]$. Define $\gamma_i$ as
$$
\gamma_i\circ \sigma_i\,:=\, \gamma\vert_{[t_{i-1}, t_i]}
$$
for every $\gamma \in {\mathcal U}.$ Hence each $\gamma_i\in \mathcal P M.$ Then the
right--hand side 
of \eqref{sympath} can be written as
\begin{equation}\label{parti}
\begin{split}
\widetilde \omega (v_1, v_2))_{(\gamma)} &= \int_{0}^{1} \omega\bigl(d {\rm ev}_t (v_1(\gamma)), d {\rm ev}_t (v_2(\gamma))\bigr) dt\\ 
&= \sum_{i=1}^{n} \int_{t_{i-1}}^{t_i} d \theta_i\bigl(d {\rm ev}_t (v_1(\gamma)), d {\rm ev}_t (v_2(\gamma))\bigr)  dt\\
&=\sum_{i=1}^{n} \int_{0}^{1} d \theta_i\bigl(d {\rm ev}_t (v_1(\gamma_i)), d {\rm ev}_t (v_2(\gamma_i))\bigr) 
\end{split}
\end{equation}
for all $\gamma \in {\mathcal U}$. Define a $1$-form $\beta$ on
${\mathcal U}\,\subset\, \mathcal P M$ as follows:
\begin{equation}\label{repara22}
\bigl(\beta(v)\bigr)({\gamma})\,=\,
\sum_{i=1}^{n}\int_{0}^{1} \theta_i\bigl(v(\gamma_i(t))\bigr)dt\, .
\end{equation}
We take note of the following identity. For a given $1$-form $\alpha$ in $M$, 
we define a $1$-form $A$ on ${\mathcal P}{M}$  by
$$\bigl(A(v)\bigr){(\gamma)} \,=\, \int_{0}^1 \alpha\bigl(d {\rm ev}_t (v(\gamma))\bigr)
dt\, , $$ 
where $v \in T {\mathcal P}{M}.$ Since the exterior derivation commutes with the pull-back
operation, we have
$$\bigl({\rm d} A (v_1, v_2)\bigr)({\gamma})\,=\, \int_{0}^{1} d \alpha
\bigl(d {\rm ev}_t (v_1(\gamma)), d {\rm ev}_t (v_2(\gamma))\bigr) dt\, .$$
Thus from \eqref{parti},
$$
{\widetilde \omega}\vert_{\mathcal U}\,=\, {\rm d}\beta\, .
$$
Therefore, we have the following proposition.

\begin{proposition}\label{p:localsympotpath}
Let $\omega$ be a symplectic form on $M$. Let $0 = t_0 < t_1 <\cdots < t_n = 1$ be a
partition of $[0, 1]$, subject to the following conditions:
\begin{enumerate}
\item for each $i \in \{1, \cdots ,n\}$, we have an open subset $U_i\subset M$
on which $\omega\vert_{U_i}\,=\,d\theta_i$, and 
\item we have ${\mathcal U} \subset \mathcal P M$, such that each $\gamma \in {\mathcal U}$
satisfies $${\rm Image}\left(\gamma\vert_{[t_{i - 1}, t_i ]}\right)\,\subset\, U_i$$
for each $i\in \{1, \cdots ,n\}$.
\end{enumerate}
 Then the symplectic form $\widetilde \omega$ restricted to $\mathcal U$ can be written as
 $$\widetilde \omega\vert_{\mathcal U}\,=\,{\rm d}\beta,$$
where $\beta$ is defined in \eqref{repara22}.
\end{proposition}

We call $\beta$ to be a \textit{potential $1$-form} of $\widetilde \omega$ on ${\mathcal U}.$

\section{Line bundle on the path space}

Let $(M,\, \omega)$ be a compact connected symplectic manifold
such that the cohomology class of $(2\pi\hbar)^{-1}\omega$ lies in the
image of $H^2(M,\, {\mathbb Z})$ in $H^2(M,\, {\mathbb R})$, where
$\hbar=h/2\pi,$ and $h$ is the Planck constant. This implies that there exists
a Hermitian line bundle $L\longrightarrow M$ and a connection $\nabla$ on $L$, whose
curvature is $\hbar^{-1}\cdot\omega$ \cite{Wo}. Locally $\nabla$ is given by
$$\nabla\,:=\,d-\sqrt{-1}\hbar^{-1}\theta\, ,$$
where $\omega=d\theta$ on some open set $U\subset M.$ In this section,
we will construct a line bundle over $\mathcal P M$ and a connection ${\widetilde \nabla}$
on that line bundle with curvature $\hbar^{-1}\cdot\widetilde \omega$, where
$\widetilde \omega$ is the symplectic form defined in \eqref{sympath}. Note that since
the connection $\nabla$
has curvature $\hbar^{-1}\cdot\omega,$ then holonomy around a  loop $\gamma$ on $M$ can be written as
$${\rm exp}\left(\frac{\sqrt{-1}}{\hbar}\int_{\Sigma}\omega\right),$$
where $\Sigma$ is an orientable surface spanning loop $\gamma$. So to have a well defined
holonomy, for any orientable closed surface ${\Sigma_{\rm cl}}$,
the form $\omega$ must satisfy the following condition~\cite{Wo}:
\begin{equation}\label{inten}
\int_{\Sigma_{\rm cl}}{\omega}\,=\,2\pi n \hbar\, ,
\end{equation}
where $n\in {\mathbb Z}.$ Here, the constant factor
 $(2\pi\hbar)^{-1}$ in the cohomology condition  has been chosen so that our construction is consistent 
with Dirac's prescription in quantum mechanics: If $A, B$ are two classical observables,
then their quantum mechanical counterparts, Hermitian operators ${\hat A}, {\hat B}$
should satisfy
$$[\widehat{A}\, , \widehat{B}]\,=\,-\sqrt{-1}\hbar \widehat{\{A,B\}},$$ 
where $\{,\, \}$ is the Poisson bracket and $[-\, , - ]$ is the commutator.

In order to construct such a line bundle on $\mathcal P M$, we will use a well-known 
identity related to  \textit {Chen integral}~\cite{chen, chen1}. Let us recall 
the definition of a first order Chen integral. Given a $p$-form
$\alpha_p$ on a manifold $M$, we define the first order Chen integral as
\begin{equation}\label{firstchen}
\begin{split}
&\left((\int_{\rm Chen} \alpha_p)(v_1,\cdots, v_{p-1})\right)({\gamma})\\ 
:=&\int_{0}^{1} \alpha_p\bigl(\stackrel{{.}}{\gamma}(t) , 
d {{\rm ev}_t}(v_1(\gamma)), \cdots, d {{\rm ev}_t}(v_{p-1}(\gamma))\bigr) dt\, ,
\end{split}
\end{equation} 
where $\gamma\,\in\,{\mathcal P} M$ and $v_1, \cdots, v_{p-1}$ are vector fields
on ${\mathcal P} M$. 
Clearly, $\int_{\rm Chen} \alpha_p$ in the left--hand side of \eqref{firstchen} is a $p-1$ form on
the path space. 
Higher order Chen integrals can be defined by  iteration. However we will only use the
first order Chen integral, so we are not going to discuss a Chen integral in its
full generality. Before proceeding further, let us clarify our notation for various integrals 
used here.
\begin{itemize}
\item We denote a first order Chen integral of a form $\alpha$ on $M$ as 
$\int_{\rm Chen} \alpha$.
\item $\int_{N}\alpha$ is the integral of a $p$-form $\alpha$
defined on $M$ over the submanifold $N\,\subset\, M$ of dimension $p$.
\item $\int_{\gamma}\alpha$ is the line integral along a path $\gamma:[0,1]
\,\longrightarrow \,X$ of a $1$-form $\alpha$ on some space $X$.
\end{itemize}
We have the following formula for the exterior derivative for a 
first order Chen integral~(see \cite[Proposition 3.1]{CLS}):
\begin{equation}\label{exterifirstchen}
{\rm d} \int^t_{\rm Chen} \alpha_p \,=\, {\rm ev}_t^* \alpha_p - {\rm ev}_0^* \alpha_p - \int^t_{\rm Chen} d \alpha_p\, ,
\end{equation}
where 
\begin{equation}\label{trancchen}
\begin{split}
&\left((\int_{\rm Chen}^t \alpha) (v_1,\cdots,v_{p-1})\right)(\gamma_t) \\
:=&\int_{0}^t \alpha\bigl(\stackrel{.}{\gamma}(\tau), d {{\rm ev}_{\tau}}(v_1(\gamma)), \cdots, d {{\rm ev}_{\tau}}(v_{p-1}(\gamma))\bigr) d \tau
\end{split}
\end{equation}
and $\gamma_t:[0,t]\longrightarrow M$, $t\in [0, 1]$.

Now consider the symplectic form $\omega$ on $M$. Since $d\omega=0$,
the identity \eqref{exterifirstchen} implies that
\begin{equation}\label{homologsympl}
{\rm d} \int^t_{\rm Chen} \omega \,=\, {\rm ev}_t^* \omega - {\rm ev}_0^* \omega\, .
\end{equation}
Integrating both sides of \eqref{homologsympl} we obtain
\begin{equation}\label{symplecglobal}
{\rm d} \lambda \,=\, {\widetilde \omega} - {\rm ev}_0^* \omega\, ,
\end{equation}
where $\lambda$ is a global $1$-form on $\mathcal P M$ defined using \eqref{trancchen} as
\begin{equation}\label{global1form}
\bigl(\lambda(v)\bigr)(\gamma) := \int_0^1 \left(\int_0^t 
\omega(\stackrel{.}{\gamma}(\tau), d {{\rm ev}_{\tau}}(v(\gamma))) d \tau\right) dt
\end{equation}
for any vector field $v$ on $\mathcal P M$; the form $\widetilde \omega$ is defined as in \eqref{sympath}
$$\bigl(\widetilde \omega(v_1, v_2)\bigr)({\gamma}) = \int_{0}^{1} \omega\bigl(d {{\rm ev}_{t}}(v_{1}(\gamma)),
d {{\rm ev}_{t}}(v_{1}(\gamma))\bigr) dt\, ,$$
with $v_1, v_2$ being vector fields on $\mathcal P M$.

\begin{lemma}\label{l:globalone}
Let $\omega$ be a symplectic form on $M$ and $\widetilde \omega$ be the symplectic form on $\mathcal P M$ defined in \eqref{sympath}.
 Then $\widetilde \omega$ can be written as a sum of pullback of $\omega$ and a globally defined exact form:
$${\widetilde \omega} \,=\, {\rm ev}_0^* \omega+{\rm d} \lambda \, ,$$
where $\lambda$ is defined in \eqref{global1form}.
\end{lemma}

Note that if we considered some other interval $[a, b]\subset {\mathbb R}$ to define the path space as 
$\{\gamma:[a, b]\longrightarrow M\}$, we still could write \eqref{symplecglobal} as
$${\rm d} \lambda^{\prime}\,=\, {\widetilde \omega}^{\prime} - {\rm ev}_a^* \omega^{\prime}
\, ,
$$
where
\begin{equation}\label{b-a}
\begin{split}
&\bigl({\widetilde \omega}^{\prime}(v_1, v_2)\bigr)({\gamma}) := \int_{a}^{b} \omega\bigl(v_1(\gamma(t)), v_2(\gamma(t))\bigr) dt\, ,\\
& \omega^{\prime}\,:=\,(b-a)\omega\, ,\\
&\bigl(\lambda^{\prime}(v)\bigr)(\gamma) \,:=\, 
\int_a^b \left(\int_a^t \omega(\stackrel{.}{\gamma}(\tau), 
v(\gamma(\tau)) d \tau\right) dt\, .
\end{split}
\end{equation}
Hence, our new path space symplectic form ${\widetilde \omega}^{\prime}$ would continue to be
the sum of the pull-back of the symplectic form 
$\omega^{\prime}$ and the exterior derivative of a globally defined $1$-form $\lambda^{\prime}$ on $\mathcal P M$.

Let $L$ be a complex line bundle on $M$, $H$ a Hermitian structure on $L$ and
$\nabla$ a Hermitian connection on $L$ such that the curvature of $\nabla$ is
$\hbar^{-1}\cdot\omega$. We noted earlier that the integrality condition on the
cohomology class of $\omega$ ensures that such a triple $(L\, ,H\, ,\nabla)$ exists.

First observe that by Lemma~\ref{l:globalone}, the form $\widetilde \omega$ is
the sum of pull-back of $\omega$ and an exact form ${\rm d}\lambda$. 
So using the map ${\rm ev}_0: \gamma \longmapsto \gamma(0)$, we first construct the
pull-back line bundle ${\rm ev}_0^* L$ on the path space ${\mathcal P}M$.
The Hermitian structure $H$ on $L$ pulls back to a Hermitian structure on
${\rm ev}_0^* L$. The Hermitian connection $\nabla$ on $L$ pulls back to a Hermitian
connection ${\rm ev}_0^*\nabla$ on ${\rm ev}_0^* L$. Since the curvature of $\nabla$
is $\hbar^{-1}\cdot\omega$, we conclude that the curvature of this connection ${\rm ev}_0^*\nabla$ is
$\hbar^{-1}\cdot{\rm ev}_0^* \omega$. Note that $\lambda$ is a globally defined real $1$-form
on $\mathcal P M$. As the space of Hermitian connections over a
complex line bundle is an affine space for the space of real $1$-forms,
\begin{equation}\label{conpullback}
\widetilde \nabla\,:=\, {\rm ev}_0^* \nabla-\sqrt{-1}\hbar^{-1}\cdot\lambda
\end{equation} 
is a Hermitian connection on the Hermitian line bundle
$({\rm ev}_0^* L\, , {\rm ev}_0^* H)$. From Lemma \ref{l:globalone} we conclude
that the curvature of $\widetilde \nabla$ is $\hbar^{-1}\cdot\widetilde \omega$.

\begin{proposition}\label{p:tildenabla}
Suppose $(L\, , H)$ be a Hermitian  line bundle on $M$ with a Hermitian
connection $\nabla$ whose curvature is $\hbar^{-1}\cdot\omega$. 
Then the Hermitian connection $\widetilde \nabla\,=\,{\rm ev}_0^* \nabla-\sqrt{-1}\hbar^{-1}\cdot\lambda$ on the Hermitian line bundle
$({\rm ev}_0^* L\, , {\rm ev}_0^* H)$ over ${\mathcal P}M$ has
curvature  $\hbar^{-1}\cdot\widetilde \omega$, where  $\lambda$ and
$\widetilde \omega$ are constructed in \eqref{global1form}
and \eqref{sympath} respectively.
\end{proposition}

We digress here for a moment to observe that, 
albeit in  entirely different contexts, connections similar to \eqref{conpullback} appear for the principal $G$ bundle 
${\mathcal P}_{\overline A}P\longrightarrow \mathcal P M$, where $\overline A$ is a connection on a principal 
$G$ bundle $P \longrightarrow M$ and ${\mathcal P}_{\overline A}P\subset {\mathcal P}P$ is the space 
of horizontally lifted paths by the connection ${\overline A}$ ~(\cite{catta},
\cite[Proposition 2.2]{CLS2}).

\begin{figure}[ht]
\begin{center}
\epsfxsize=3.0in \epsfysize=2.5in
\rotatebox{0}{\epsfbox{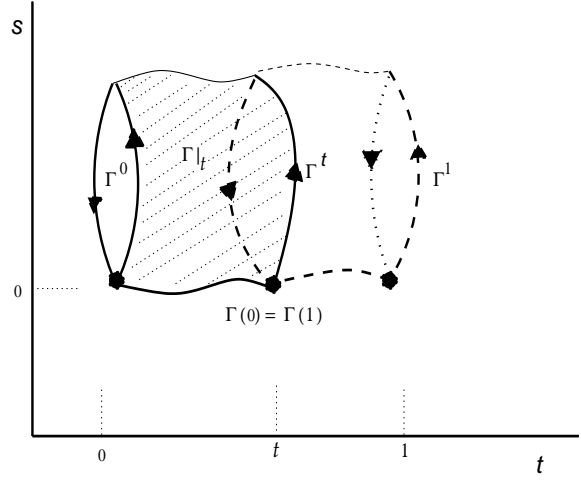}}
\caption{A loop $\Gamma$ on path space}
\end{center}
\end{figure}\label{f:paraclose2}

Here we give an explicit expression for the holonomy of the connection ${\widetilde \nabla}$. 
Consider a ``path on path space" ${\mathcal P} M$, defined as 
$$
\Gamma\,:\,[0,1] \,\longrightarrow\, {\mathcal P} M\, .
$$
Hence 
$$
\Gamma\,:\, [0,1]\times [0,1]\,\longrightarrow\, M\, , \ 
(s, t)\longmapsto \Gamma(s,t):=\Gamma_s(t):=\Gamma^{t}(s)\, .
$$
We call the path  $\Gamma_s:[0, 1]\longrightarrow M$, obtained by fixing the first coordinate $s\in [0, 1]$, a \textit{transversal path}
and the path  $\Gamma^t:[0, 1]\longrightarrow M$, obtained by fixing the second coordinate $t\in [0, 1]$, a \textit{longitudinal path}. 
We define the ``line integral'' of the $1$-form $\lambda$ on $\mathcal P M$, given by \eqref{global1form}, along 
the path on path space $\Gamma$ to be
\begin{equation}\label{pathlinelambda}
\begin{split}
\int_{\Gamma}\lambda\,&:=\,\int_{0}^1 \lambda (\Gamma^{\prime}(s)) ds\\
&=\,\int_{0}^1 \left( \int_{0}^1 dt \int_{0}^ t \omega ( \stackrel{.}
{\Gamma}(s, \tau), \Gamma^{\prime} (s, \tau)) d \tau\right)ds\, ,
\end{split}
\end{equation}
where 
 $ \stackrel{.}{\Gamma}(s, \tau):= \frac{\partial \Gamma(s, \tau)}{\partial \tau}$ and $\Gamma^{\prime} (s, \tau) := \frac{\partial \Gamma(s, \tau)}{\partial s}$. 
As integrations with respect to $s$ and $t$ are interchangeable, we can write \eqref{pathlinelambda} as
$$\int_{0}^1 \left(\int_{0}^1 ds \int_{0}^ t \omega ( \stackrel{.}{\Gamma}(s, 
\tau), \Gamma^{\prime} (s, \tau)) d \tau\right)dt\, .$$
Observe that 
$$\int_{0}^1 ds \int_{0}^ t \omega ( \stackrel{.}{\Gamma}(s, \tau), 
\Gamma^{\prime} (s, \tau)) d \tau\,=\,\int_{\Gamma\vert_t}\omega\, ,$$
where $\Gamma\vert_t\,:\,[0, 1]\times [0, t]\,\longrightarrow
\,M$. Thus \eqref{pathlinelambda} reads as
\begin{equation}\label{lambdasimple}
\int_{\Gamma}\lambda\,=\,\int_0^1 dt\int_{\Gamma\vert_t}\omega\, .
\end{equation}
A ``loop'' on the path space is given by a map
$$
\Gamma: [0,1] \longrightarrow {\mathcal P}M
$$
such that $\Gamma(0) \,=\, \Gamma (1)$.
Consequently, each longitudinal path $\Gamma^t:[0, 1]\longrightarrow M, t\in [0, 1]$, is a loop on $M$
based at the point 
$$\Gamma^t(0)\,=\,\bigl(\Gamma(0)\bigr)(t)\,=\,\Gamma^t(1)
\,=\,\bigl(\Gamma(1)\bigr)(t)\, .$$

Now let ${\rm G}$ be a parametrized ``surface" on path space $\mathcal P M$:
\begin{equation}\label{surfacepm}
\begin{split}
{\rm G}&:[0, 1]\times [0, 1]\longrightarrow \mathcal P M,\\
&(\sigma, \epsilon)\longmapsto {\rm G}(\sigma, \epsilon)\in \mathcal P M,\\
\Rightarrow &\bigl({\rm G}(\sigma, \epsilon)\bigr)(t)\in M, \quad t\in [0, 1].
\end{split}
\end{equation}
So, given such a surface $\rm G$ on $\mathcal P M$, for each $t\in [0, 1]$ we have a
surface ${\rm G}_t:[0, 1]\times [0, 1]\longrightarrow M$ on $M$ defined by
\begin{equation}\label{surpmtom}
{\rm G}_t(\sigma, \epsilon)\,:=\,\bigl({\rm G}(\sigma, \epsilon)\bigr)(t)\, .
\end{equation}
A surface $\rm G$ on $\mathcal P M$  spans a loop $\Gamma$ on $\mathcal P M$, when for each $t\in [0, 1]$ the associated
surface ${\rm G}_t$ on $M$ spans the loop $\Gamma^t$ on $M$. Orientation of each $\Gamma^t$ induces an orientation 
on each ${\rm G}_t.$

\begin{proposition}
Let $\theta$ be a symplectic potentials for the symplectic form $\omega$ on $M$; that is locally
$\omega\,=\,d\theta$, and let
$\nabla$ be a connection on the line bundle $L\,\longrightarrow \,M$ with
curvature $\hbar^{-1}\omega$.
Let  $\widetilde \nabla={\rm ev}_0^* \nabla-\sqrt{-1}\hbar^{-1}\lambda$ be the
connection on the line bundle ${\rm ev}_0^* L\,\longrightarrow \,
{\mathcal P}M$, where $\lambda$ as defined in \eqref{global1form}. Then
the holonomy of $\widetilde \nabla$ around a loop $\Gamma$ on $\mathcal P M$ is:
\begin{enumerate}
\item{when $\Gamma$ is entirely contained in a single domain of $\theta$:
\begin{equation}\label{holopath}
{\rm exp}\left(\frac{{\sqrt{-1}}}{\hbar}\int_0^1 dt \oint_{\Gamma^t}\theta\right),
\end{equation}
}
\item{more generally,
\begin{equation}\label{holopathgen}
{\rm exp}\left(\frac{{\sqrt{-1}}}{\hbar}\int_0^1 dt (\int_{{\Sigma_t}}\omega)\right),
\end{equation}
where $\Sigma_t$ is an oriented surface spanning $\Gamma^t$,
is independent of the choice of $\Sigma_t$ because of the condition
in \eqref{inten}.}
\end{enumerate}
\end{proposition}

\begin{proof}
\mbox{}
\begin{enumerate}
\item{
Assume that $\Gamma$ is contained within a single domain of $\theta$. Note by Stokes' theorem (see Figure 1),
$$\int_{\Gamma\vert_t}\omega\,=\,-\oint_{\Gamma^0}\theta+\oint_{\Gamma^t}
\theta\, .$$
Thus \eqref{lambdasimple} implies
\begin{equation}\label{lamdaloop}
\int_{\Gamma}\lambda=-\oint_{\Gamma^0}\theta+\int_0^1dt\oint_{\Gamma^t}\theta.
\end{equation}
On the other hand the contribution due to the ${\rm ev}_0^* \nabla$ part is 
\begin{equation} \label{pullnablahol}
\begin{split}
&\frac{{\sqrt{-1}}}{\hbar}\int_0^1\bigl({\rm ev}_0^*\theta\bigr)(\Gamma^{\prime}(s))ds\\
=&\frac{{\sqrt{-1}}}{\hbar}\int_0^1\theta(\Gamma^{\prime}(s,0))ds\\
=&\frac{{\sqrt{-1}}}{\hbar}\int_0^1\theta(\Gamma^{\prime 0}(s))ds=\frac{\sqrt{-1}}{\hbar}\oint_{\Gamma^0}\theta.
\end{split}
\end{equation}
Hence adding \eqref{lamdaloop} and \eqref{pullnablahol} on the exponent, we have 
$${\rm exp}\left(\frac{{\sqrt{-1}}}{\hbar}\int_0^1 dt \oint_{\Gamma^t}\theta\right)$$ 
for the holonomy of ${\widetilde \nabla}$ along $\Gamma$. 
}
\item{
Now consider the general case where $\Gamma$ is not contained within a single domain of $\theta.$ 
Suppose that $\rm G$ is a surface on $\mathcal P M$ which spans the loop $\Gamma$ on $\mathcal P M,$ 
as described in \eqref{surfacepm} and \eqref{surpmtom}. We have already established in 
Proposition~\ref{p:tildenabla} that the curvature of the connection $\widetilde \nabla$ 
is $\hbar^{-1}\widetilde \omega.$ The holonomy around $\Gamma$ for the connection
$\widetilde \nabla$ is
$${\rm exp}\left(\sqrt{-1}{\hbar}^{-1}\iint\widetilde \omega(\partial_{\sigma}{\rm G}, \partial_{\epsilon}{\rm G})d\sigma d\epsilon\right),$$
where $\partial_{\sigma}{\rm G}:=\frac {\partial {\rm G}(\sigma, \epsilon)}{\partial \sigma}$ and $\partial_{\epsilon}{\rm G}:=\frac {\partial {\rm G}(\sigma, \epsilon)}{\partial \epsilon}.$
Using definition of $\widetilde \omega$ in \eqref{sympath}, the expression in the exponent can be written as:
\begin{equation}\label{curvint}
\begin{split}
&\sqrt{-1}{\hbar}^{-1}\int_0^1 dt \left(\iint\omega((\partial_{\sigma}{\rm G})(t), (\partial_{\epsilon}{\rm G})(t))d\sigma d\epsilon\right)\\
=&\sqrt{-1}{\hbar}^{-1}\int_0^1dt \left(\iint\omega(\partial_{\sigma}{\rm G}_t, \partial_{\epsilon}{\rm G}_t)d\sigma d\epsilon\right)\, ;
\end{split}
\end{equation}
in the last line we have used \eqref{surpmtom}. But, by construction $\rm G$ spans the loop $\Gamma$ on $\mathcal P M$, hence each surface ${\rm G}_t$ on $M$ spans the 
loop $\Gamma^t$ on $M$. So, we have as a consequence of \eqref{inten}
\begin{equation}
\iint\omega(\partial_{\sigma}{\rm G}_t, \partial_{\epsilon}{\rm G}_t)d\sigma d\epsilon=2\pi n \hbar+\int_{\Sigma_t}\omega\, ,
\end{equation}
where $\Sigma_t$ is an arbitrary orientable surface on $M$ spanning $\Gamma^t$ and $n\in {\mathbb Z}$.
Therefore, from \eqref{curvint} we conclude the required expression
\begin{equation}\nonumber
{\rm exp}\left(\frac{{\sqrt{-1}}}{\hbar}\int_0^1 dt (\int_{{\Sigma_t}}\omega)\right).
\end{equation}
}
\end{enumerate}
This completes the proof.
\end{proof}

Note again, that particular choice of the interval $[0, 1]\subset \mathbb R$ to define $\mathcal P M$ is not significant here. It is apparent, from 
\eqref{b-a}, that any other choice of an interval $[a, b]\subset \mathbb R$ would be as good and in that case \eqref{holopath} would have been 
${\rm exp}\left(\frac{{\sqrt{-1}}}{\hbar}\int_a^b dt \int_{\Sigma_t}\omega\right).$

{\vskip 0.3 cm}
Let the symplectic manifold $M$ be of dimension $2m$. Assume that $\omega$ satisfies 
the earlier mentioned integrality condition. Then there exists a Hermitian line bundle
$(L\, , H)$ on $M$ equipped with a Hermitian connection with curvature
${\hbar}^{-1}\omega$. Following \cite{Wo}
we call it a \textit {pre-quantum bundle}. Let ${\mathbf S}$ be the space of
square integrable sections $s:M\longrightarrow L$ of the line bundle $L$, with
the inner product given by
$$
\langle{s}_1, {s}_2\rangle \,:=\, \int_M \langle s_1, s_2\rangle_H \omega^m\, ,
$$
where $s_1, s_2\in {\mathbf S}$, $\omega^m$ is the volume form on $M$ and  
$\langle, \rangle_{H}$ is the inner product with respect to the Hermitian
metric $H$. The vector space ${\mathbf S}$ equipped with the above
inner product is a Hilbert 
space. Now, given a section ${\mathbf S}\ni s:M\longrightarrow L$, we have a 
natural pull-back section ${\rm ev}_0^{*}s:={\widetilde s}$ on the 
pull-back line bundle ${\rm ev}_0^* L\,\longrightarrow\, {\mathcal P}M$:
$${\widetilde s}\,:\, {\mathcal P}M\,\longrightarrow \,
{\rm ev}_0^* L\, .$$
Let $\widetilde {\mathbf S}$ be the space of all such sections:
$$
\widetilde {\mathbf S}\,:=\, \{{\rm ev}_0^{*}s\,\mid\, s\in {\mathbf S}\}\, .
$$
The inner product in ${\mathbf S}$ naturally induces
a inner product in  $\widetilde {\mathbf S}$ given by
\begin{equation}\label{innerpath}
\langle\widetilde{s}_1, \widetilde{s}_2\rangle\,:=\,
\int_M \langle s_1, s_2\rangle_H \omega^m\, ,
\end{equation}
where $\widetilde {\mathbf S}\ni \widetilde{s}_1= {\rm ev}_0^{*}s_1$ and 
$\widetilde {\mathbf S}\ni \widetilde{s}_2= {\rm ev}_0^{*}s_2$, $s_1, s_2\in {\mathbf S}$.
Thus $\widetilde {\mathbf S}$ defines a Hilbert space. 

Let us recall the space of functions $\Omega^{0} {\mathcal P}{M}$ on path space  
defined in \eqref{uptok}. In particular, $\Omega_1^{0} {\mathcal P}{M}\subset \Omega^{0} {\mathcal P}{M}$
is the space of $C^{\infty}$ functions on $\mathcal P M$ which can be written as:
$$
\widetilde{f} (\gamma) \,=\, \int_{0}^{1} f (\gamma(t)) dt
$$
for all $\gamma \,\in\, \mathcal P M,$ where $f\,\in\, C^{\infty} (M)$. 

Let us associate an operator ${\phi}_{\rm op}$ with each 
${\phi}\in \Omega^{0} {\mathcal P}{M}$, given by
$$
{\phi}_{\rm op}\,:=\,-\sqrt{-1}\hbar{\widetilde \nabla}_{{\widetilde X}_{\phi}}+{\phi}\, ,
$$
where ${\widetilde \nabla}$ is the connection on the line bundle
${\rm ev}_0^* L\longrightarrow {\mathcal P}M$
defined in Proposition~\ref{p:tildenabla} and ${\widetilde X}_{\phi}\in T\mathcal P M$
is the Hamiltonian vector field
corresponding to ${\phi}$ defined in \eqref{hamilpm}. Let ${\mathbb F}$ be the space of all such 
operators obtained from the functions in $\Omega^{0} {\mathcal P}{M}$. It is a
straightforward verification that
$$
[\phi_{1\,{\rm op}}, \phi_{2 \, {\rm op}}]\,=\,
-{\sqrt{-1}}\hbar\{\phi_1, \phi_2\}_{\rm op}\, ,
$$
where $\phi_1, \phi_2\in \Omega^{0}\mathcal P M$ and $\{\phi_1, \phi_2\}$ is as defined
in \eqref{pathpoiss}. We observe that
${\mathbb F}$ has the following action on the Hilbert space ${\widetilde {\mathbf S}}$
$$
{\phi}_{\rm op}{\widetilde s}\,:=\,
\bigl(-\sqrt{-1}\hbar{\widetilde \nabla}_{{\widetilde X}_{\phi}}+{\phi}\bigr){\widetilde s},
$$
where ${\widetilde s}\in {\widetilde {\mathbf S}}$ and ${\phi}_{\rm op}\in {\mathbb F}$. Note that by
construction,
any ${\widetilde s}\in {\widetilde {\mathbf S}}$ is pull-back of some $s\in {\mathbf S}$.

Let ${\widetilde f}\in \Omega_1^{0} {\mathcal P}{M}\subset \Omega^{0} {\mathcal P}{M}$ be given by 
$${\widetilde f}(\gamma)\,=\,\int_{0}^{1} f (\gamma(t)) dt\, ,\,\ f\,\in\, C^{\infty} (M)\, .$$
Then the corresponding operator
 ${\widetilde f}_{\rm op}\in {\mathbb F}$ is given by 
$$
{\widetilde f}_{\rm op}\,:=\,-\sqrt{-1}\hbar{\widetilde \nabla}_{{\widetilde X}_{\widetilde f}}+
{\widetilde f}\, .
$$
Recall by Proposition~\ref{p:identify},
\begin{equation}\nonumber
\bigl({{\widetilde X}_{\widetilde f}}(\gamma)\bigr)(t)=X_f(\gamma(t)).
\end{equation}
Let ${\widetilde s}={\rm ev}_0^*s.$ Then using definition of ${\widetilde \nabla}$ in Proposition~\ref{p:tildenabla} we obtain the following 
expression for $\bigl({\widetilde f}_{\rm op}{\widetilde s}\bigr)$
\begin{eqnarray}
\bigl({\widetilde f}_{\rm op}{\widetilde s}\bigr)(\gamma)&=
&-\sqrt{-1}\hbar\bigl({\nabla_{X_f}s}\bigr)(\gamma(0))\nonumber\\
&&-\left(\int_0^1 \left(\int_0^t \omega(\stackrel{.}{\gamma}(\tau), {{ X}_{f}}(\gamma(\tau)) d \tau\right)dt\right)s(\gamma(0))\nonumber\\
&&+\left(\int_0^{1}f\bigl(\gamma(t)\bigr)dt \right)s(\gamma(0)).\nonumber
\end{eqnarray}
 
Now, suppose ${\zeta}\,:=\,{\widetilde f}{\widetilde g}\,\in\,
\Omega^{0} {\mathcal P}{M}$, where
${\widetilde f}, {\widetilde g}\in \Omega_1^{0} {\mathcal P}{M}$ are given by
$$
{\widetilde f}(\gamma)=\int_{0}^{1} f (\gamma(t)) dt\, ,\
{\widetilde g}(\gamma)=\int_{0}^{1} f (\gamma(t)) dt\, ,
$$
where $f\, ,g\,\in\, C^{\infty} (M)$.
Using \eqref{hamilsum} and the identity 
$${\widetilde \nabla}_{{\widetilde X}_{\zeta}}\,=\,
{\widetilde g}{\widetilde \nabla}_{{\widetilde X}_{\widetilde f}}+{\widetilde f}
{\widetilde \nabla}_{{\widetilde X}_{\widetilde g}}\, ,$$
it is straightforward to show that:
 \begin{eqnarray}
&&\bigl({\zeta}_{\rm op}{\widetilde s}\bigr)(\gamma)=\nonumber\\
&&-\sqrt{-1}\hbar\bigl(\int_{0}^{1} g (\gamma(t)) dt\bigr)\bigl({\nabla_{X_f}s}\bigr)(\gamma(0))\nonumber\\
&-&\sqrt{-1}\hbar\bigl(\int_{0}^{1} f (\gamma(t)) dt\bigr)\bigl({\nabla_{X_g}s}\bigr)(\gamma(0))\nonumber\\
&-&\bigl(\int_{0}^{1} g (\gamma(t)) dt\bigr)\left(\int_0^1 \left(\int_0^t \omega(\stackrel{.}{\gamma}(\tau), {{X}_{f}}(\gamma(\tau)) d \tau\right)dt\right)s(\gamma(0))\nonumber\\
&-&\bigl(\int_{0}^{1} f (\gamma(t)) dt\bigr)\left(\int_0^1 \left(\int_0^t \omega(\stackrel{.}{\gamma}(\tau), {{X}_{g}}(\gamma(\tau)) d \tau\right)dt\right)s(\gamma(0))\nonumber\\
&&+\left(\int_0^{1}f\bigl(\gamma(t)\bigr)dt \right)\left(\int_{0}^{1} g (\gamma(t)) dt\right)s(\gamma(0)).\label{explictactzeta}
\end{eqnarray}
Thus \eqref{explictactzeta} generates the action of entire ${\mathbb F}$ on $\widetilde {\mathbf S}$.

\section{symplectic form on path space over the space of solutions}\label{s:examp}

Let $M$ be the space-time manifold of dimension $n+1$, where $n$ is the number of space directions, 
equipped with a pseudo-Riemannian metric $g$. Coordinate system on $M$ will be denoted as 
$$x^\mu, \quad \mu=0, 1,\cdots , n, $$
where $x^0$ is the time direction. Let ${\overline M}:=M \times S^1$ be the manifold with an extra spatial
dimension $S^1.$ Let $S^1$ be parametrized by an angle $\theta \in [0, 2\pi]$ and has a radius $R.$ Thus a coordinate
system ${\overline x}^A, A=0, 1,\cdots , n+1$, in ${\overline M}$ is given by
\begin{equation}\label{extraspatcoord}
\begin{split}
&{\overline x}^{\mu}=x^\mu, \quad \mu= 0, 1,\cdots , n,\\
&{\overline x}^{n+1}=\theta.
\end{split}
\end{equation}
The basic object of interest in this section will be collections of fields on $M$ and ${\overline M},$ which satisfy certain 
equation on $M$ and ${\overline M}$ respectively. Here we will only consider Klein-Gordon equation on $M$ and ${\overline M}.$
However, the methodology we employ is general enough
to work for other cases.

Let $\nabla$ be a connection on $M$, and let $\nabla_{\mu}:=\nabla_{\frac{\partial}{\partial x_\mu}}$
denote the covariant derivative along 
${\frac{\partial}{\partial x_\mu}}$. 
Let $g_{\mu \, \nu}\,:=\, g({\frac{\partial}{\partial x_\mu}}, {\frac{\partial}{\partial x_\nu}})$
be the components of metric $g$ on $M$. Components of the 
inverse of metric $g$ will be denoted as $\{g^{\mu \, \nu}\}$. Metric $g$
induces a metric ${\overline g}$ on ${\overline M}\,=\,M\times S^1$, which in
components $\{{{\overline g}_{A\, B}}\}, A, B=0,\cdots , n+1$ is given by:
\begin{equation}\label{gab}
\begin{split}
&{\overline g}_{\mu \, \nu}({\overline x}):=g_{\mu \, \nu}(x), \quad \mu, \nu=0,\cdots , n\\
&{\overline g}_{\mu \, n+1}({\overline x}):=0, \quad \mu=0,\cdots , n\\
&{\overline g}_{n+1\, n+1}({\overline x}):=-R^2,
\end{split}
\end{equation}
where $R$ as before
is the radius of the circle $S^1.$ Similar or somewhat more general type of metrics have already been used
in studying the graviton-scalar sector for the Kaluza-Klein theory ~\cite{ap-ch, ba-love, Fr}.   
 Let ${\overline \nabla}$ be a connection
on ${\overline M}$ such that 
$${\overline \nabla}_\mu=\nabla_{\mu}, \quad \mu= 0, 1,\cdots , n,$$
where ${\overline \nabla}_\mu={\overline \nabla}_{\frac{\partial}{\partial {\overline x}_\mu}}.$
Note that $\overline \nabla$ and $\nabla$ can be chosen to be the metric connections 
corresponding to ${\overline g}$ on $\overline M$ and $g$ on $M$ respectively.
 
Now Consider the Klein-Gordon equation on $M$:
\begin{equation}\label{KGM}
{\Box \psi} + \rho^2\psi\,=\,0\, ,
\end{equation}
where $\rho$ is a real constant, and $\Box\,:=\,-g^{\mu \nu}\nabla_{\mu}\nabla_{\nu}$. Let ${\mathcal M}$
be the space of  solutions of \eqref{KGM}:
\begin{equation}\label{solspacem}
{\mathcal M}\,:=\, \{\psi: M\longrightarrow {\mathbb R}\,\mid\, {\hskip 0.1 cm} {\Box \psi} + \rho^2\psi=0\}.
\end{equation}
This ${\mathcal M}$ has a differential structure and since ${\mathcal M}$ is a vector space, 
tangent vector space $T_{\psi}{\mathcal M}$ at any $\psi\in {\mathcal M}$ can be identified
with ${\mathcal M}$. Also, ${\mathcal M}$ has a symplectic structure
\cite[Chapter 7]{Wo}; this symplectic form $\omega$ on ${\mathcal M}$ 
is
\begin{equation}\label{symsolm}
\omega(\psi_1, \psi_2)\,:=\,
\frac{1}{2}\int_{\Sigma}\bigl(\psi_2\nabla_{\mu}\psi_1-\psi_1\nabla_{\mu}\psi_2\bigr){\rm n}^{\mu}d\sigma
\, ,
\end{equation}
where $\psi_1, \psi_2\in {\mathcal M}$ and $\Sigma \subset M$ is a Cauchy surface with volume element $d\sigma$ and unit normal vector
$\{{\rm n}^\mu\}.$ It can be shown that the right-had side of \eqref{symsolm} is actually independent of choice of the Cauchy surface $\Sigma$
\cite[Chapter 7.2]{Wo}.

Let ${\mathcal P}{\mathcal M}$ denote the path space, as defined in Section~\ref{s:poiss},
for ${\mathcal M}$. For $\gamma_1, \gamma_2\in {\mathcal P}{\mathcal M}$, we have the
pointwise addition $(\gamma_1+\gamma_2)(t):=\gamma_1(t)+\gamma_2(t)$, 
$t\in [0, 1]$. Therefore, ${\mathcal P}{\mathcal M}$ is also a vector space.

Let $\overline {\mathcal M}$ be the space of solutions of the Klein-Gordon
equation on ${\overline M}\,=\,M\times S^1$:
\begin{equation}\label{KGbarM}
{{\overline \Box} \Psi} + \rho^2\Psi\,=\,0\, ,
\end{equation}
where ${\overline \Box}\,:=\,-g^{A \, B}\nabla_{A}\nabla_{B}$, and $g^{A \, B}$ are the components
of inverse of $\overline g$ defined in \eqref{extraspatcoord}.
Thus
\begin{equation}\label{solspacebarm}
{\overline {\mathcal M}}:=\{\Psi: {\overline M}\longrightarrow {\mathbb R}\,
\mid\, {\hskip 0.1 cm} {{\overline \Box} \Psi} + \rho^2\Psi=0\}.
\end{equation}

\begin{proposition}\label{p:barmpm}
Let ${{\mathcal M}}$ and ${\overline {\mathcal M}}$ are as defined in \eqref{solspacem} and \eqref{solspacebarm} respectively. Then
${\overline {\mathcal M}}$ can be identified with a subspace of the path space  ${\mathcal P}{\mathcal M}$
 over ${{\mathcal M}}$.
\end{proposition}
\begin{proof}
We first make the following observation. Let  $\Psi({\overline x}^A)=\Psi(x^\mu,\theta)$ be a solution of \eqref{KGbarM}. 
Then for a fixed $\theta' \in [0, 2\pi]$, the function $\Psi(x^\mu,\theta')$ is a solution of
\eqref{KGM}. Then given any $\Psi\in {\overline {\mathcal M}}$
we define for each $\theta \in [0, 2\pi]$
\begin{equation}\label{pmbarm}
\Psi(x^\mu, \theta):=\left({\overline \Psi}(\frac {\theta}{2\pi})\right)(x^\mu).
\end{equation}
Hence, for every $\theta\in [0, 2\pi]$, ${\overline \Psi}(\frac {\theta}{2\pi})$ is in ${\mathcal M}$,
in other words, 
$${\overline \Psi}:[0, 1]\longrightarrow {\mathcal M} \Rightarrow {\overline \Psi}\in {\mathcal P}{\mathcal M}.$$
On the other hand, for each $t\in [0, 1]$, we define $\theta:=2\pi t$, and for a
$\gamma'\in {\mathcal P}{\mathcal M}$,
\begin{equation}\label{barmpm}
(\gamma'(t))(x^\mu)\,:=\,\Psi' (x^\mu, \theta)\, .
\end{equation}
If $\Psi'(x^\mu, \theta)$ is a solution of \eqref{KGbarM} then we identify $\gamma'$ with $\Psi'\in {{\mathcal M}}.$
Hence by \eqref{pmbarm} and \eqref{barmpm} we identify ${\overline {\mathcal M}}$ with
a subspace of ${\mathcal P}{\mathcal M}$.
\end{proof}

Following \eqref{sympath} we define the symplectic form on ${\mathcal P}{\mathcal M}$ as
\begin{equation}\label{sympathsol}
\widetilde \omega (\gamma_1, \gamma_2)\,:=\,2\pi R\int_{0}^{1} \omega \bigl(\gamma_1(t), \gamma_2(t)\bigr)dt\, ,
\end{equation}
where $\gamma_1, \gamma_2\in {\mathcal P}{\mathcal M}$, and $\omega$ is the symplectic form on ${\mathcal M}$ defined in \eqref{symsolm}.
Next we show that $\widetilde \omega$ restricted to ${\overline {\mathcal M}}\subset {\mathcal P}{\mathcal M}$ coincides with the symplectic form 
$\overline \omega$ obtained directly from the solution space ${\overline {\mathcal M}}$. Thus  $\overline \omega$ is given by 
\begin{equation}\label{symbarm}
{\overline \omega}(\Psi_1, \Psi_2)\,:=\,
\frac{1}{2}\int_{\overline \Sigma}\bigl(\Psi_2{\overline \nabla}_{A}\Psi_1-
\Psi_1{\overline \nabla}_{A}\Psi_2\bigr){\overline{\rm n}}^{A}d{\overline \sigma}\, ,
\end{equation} 
where $\Psi_1, \Psi_2\in {\overline {\mathcal M}}$, ${\overline \Sigma} \subset {\overline M}$ is a Cauchy surface with volume element $d{\overline\sigma}$ and unit normal vector $\{{\overline{\rm n}}^A\}.$ Again, $\overline \omega$ does not depend on the choice of the Cauchy surface ${\overline \Sigma} \subset {\overline M}.$
Putting \eqref{symsolm} in \eqref{sympathsol} we obtain
\begin{equation}\label{sympathsolpmexpli}
\widetilde \omega (\gamma_1, \gamma_2)\,=\,2\pi R\int_{0}^{1} Z(t)dt\, ,
\end{equation}
where
$$
Z(t)\,=\,
\frac{1}{2}\int_{\Sigma}\bigl((\gamma_2(t))(x)(\nabla_{\mu}\gamma_1(t))(x)-
(\gamma_1(t))(x)(\nabla_{\mu}\gamma_2(t))(x)\bigr){\rm n}^{\mu}d\sigma\, ,
$$
$\Sigma \subset M$ is a Cauchy surface with volume element $d\sigma$ and unit normal vector
$\{{\rm n}^\mu\}.$ By change of variable $t\,=\,
\frac{\theta}{2\pi}$, the right--hand side can be written as
\begin{equation}\nonumber
\begin{split}
&\widetilde \omega (\gamma_1, \gamma_2)\,= \, R\int_{0}^{2\pi} Z'(\theta)d\theta,\\
&{\rm where}\\
&Z'(\theta)=\\
&\frac{1}{2}\int_{\Sigma}\left(\bigl(\gamma_2(\frac{\theta}{2\pi})\bigr)(x)\bigl(\nabla_{\mu}\gamma_1(\frac{\theta}{2\pi})\bigr)(x)-\bigl(\gamma_1(\frac{\theta}{2\pi})\bigr)(x)\bigl(\nabla_{\mu}\gamma_2(\frac{\theta}{2\pi})\bigr)(x)\right){\rm n}^{\mu}d\sigma.
\end{split}
\end{equation}
If $\gamma_1,\gamma_2 \in {\overline {\mathcal M}}\subset {\mathcal P}{\mathcal M}$, then by Proposition~\ref{p:barmpm}, we have 
$\gamma_1(\frac {\theta}{2\pi})(x)=\Psi_1(x,\theta)$ and $\gamma_2(\frac {\theta}{2\pi})(x)=\Psi_2(x,\theta)$, where $\Psi_1, \Psi_2\in {\overline {\mathcal M}}.$
Hence  
\begin{equation}\label{ccordintrans}
\begin{split}
&\widetilde \omega (\Psi_1, \Psi_2)\,=\\
&R\int_{0}^{2\pi} \left(\frac{1}{2}\int_{\Sigma}\bigl(\Psi_2(x, \theta)\nabla_{\mu}\Psi_1(x, \theta)-\Psi_1(x, \theta)\nabla_{\mu}\Psi_2(x,\theta)\bigr){\rm n}^{\mu}d\sigma\right)d\theta.
\end{split}
\end{equation}
Observe that if $\Sigma \subset M$ is a Cauchy surface with volume element $d\sigma$ and unit normal vector
$\{{\rm n}^\mu\},$ then ${\overline \Sigma}':=\Sigma \times S^1 \subset {\overline M}=M\times S^1$ is a Cauchy surface of ${\overline M}=M\times S^1.$
 Volume form on ${\overline \Sigma}':=\Sigma \times S^1$ is $d{\overline\sigma}'\,
=\,R d\sigma d\theta$ and the unit normal vector is given by $\{{\overline{\rm n}^{'\, A}}\}=\{{\rm n}^\mu, 0\}.$
Thus \eqref{ccordintrans} can be written as:
\begin{equation}\label{final}
\begin{split}
&\widetilde \omega (\Psi_1, \Psi_2)\,=\\
&\frac{1}{2}\int_{{\overline \Sigma}'}\bigl(\Psi_2(x, \theta){\overline \nabla}_{A}\Psi_1(x, \theta)-\Psi_1(x, \theta){\overline \nabla}_{A}\Psi_2(x,\theta)\bigr){\overline{\rm n}}^{' \, A}d{\overline\sigma}'\\
=&\frac{1}{2}\int_{{\overline \Sigma}'}\bigl(\Psi_2(\overline x){\overline \nabla}_{A}\Psi_1(\overline{x})-
\Psi_1(\overline{x}){\overline \nabla}_{A}\Psi_2(\overline{x})\bigr){\overline{\rm n}}^{' \, A}d{\overline \sigma}'\, .
\end{split}
\end{equation}
Since $\overline \omega$ in \eqref{symbarm} does not depend on the choice of the Cauchy
surface $\overline \Sigma$, we conclude the following by comparing
the right--hand sides of \eqref{symbarm} and \eqref{final}.

\begin{proposition}\label{p:coinsym}
By Proposition~\ref{p:barmpm} we identify ${\overline {\mathcal M}}$ with a subspace of ${\mathcal P}{\mathcal M}$. Then the restriction of the symplectic form 
${\widetilde \omega}$, defined in \eqref{sympathsol}, to ${\overline {\mathcal M}}\subset {\mathcal P}{\mathcal M}$ coincides with the symplectic form ${\overline \omega}$ given in \eqref{symbarm}. 
\end{proposition}

\section*{Acknowledgments}
Authors thank the anonymous referee for carefully reading the manuscript. Chatterjee acknowledges fellowship from the \textit {Jacques Hadamard Mathematical Foundation}.

\end{document}